\documentclass[12pt,oneside]{amsart} 
\usepackage{amssymb,amscd}
\usepackage{euscript}

      \theoremstyle{plain}
      \newtheorem{theorem}{Theorem}[section]
      \newtheorem{lemma}[theorem]{Lemma}
      \newtheorem{corollary}[theorem]{Corollary}
      \newtheorem{proposition}[theorem]{Proposition}

\numberwithin{equation}{section}

      \makeatletter
      \def\@setcopyright{}
      \def\serieslogo@{}
      \makeatother

\def\E{\mathcal{E}}
\def\F{\tilde F}
\def\G{\mathcal{G}}
\def\M{\mathcal{M}}
\def\W{\mathcal{W}}
\def\O{\mathcal O}
\def\c{\EuScript{C}}
\def\R{\mathbb R}
\def\rd{{\mathbb R ^d}}   
\def\C{\mathbb C}
\def\Z{\mathbb Z}
\def\N{\mathbb N}

\def\dist{\text{dist}}
\def\diam{\text{diam}}
\def\Id{\text{Id}}
\def\e{\epsilon}
\def\a{\alpha}
\def\ta{\tilde \alpha}
\def\b{\beta}

\def\QED{\hfill\hfill{\square}}

\begin{document}

\author{Boris Kalinin$^\ast$ and Victoria Sadovskaya$^{\ast\ast}$}

\address{Department of Mathematics $\&$ Statistics, 
 University of South  Alabama, Mobile, AL 36688, USA}
\email{kalinin@jaguar1.usouthal.edu, sadovska@jaguar1.usouthal.edu}

\title [Linear cocycles over hyperbolic systems $\;\;$]
{Linear cocycles over hyperbolic systems and criteria of conformality} 

\thanks{$^{\ast}$  Supported in part by NSF grant DMS-0701292}
\thanks{$^{\ast\ast}$ Supported in part by NSF grant DMS-0901842}


\begin{abstract}
In this paper we study H\"older continuous linear cocycles over 
transitive Anosov diffeomorphisms. 
Under various conditions of relative pinching we establish 
properties including existence and continuity of
measurable invariant sub-bundles and conformal structures.
We use these results to obtain criteria for cocycles to be 
isometric or conformal in terms of their periodic data. We show
that if the return maps at the periodic points are, in a sense,
conformal or isometric then so is the cocycle itself with respect
to a H\"older continuous Riemannian metric.
\end{abstract}

\maketitle 

 
 \section{Introduction}

Linear cocycles over a dynamical system $f: \M \to \M$ appear naturally 
in various areas of dynamics and applications. Examples  include 
derivative cocycles as well as stochastic processes and  random matrices. 
A linear cocycle over $f$ is an automorphism $F$ of a vector 
bundle $\E$ over $\M$ that projects to $f$.
In the case  of a trivial vector bundle $\M \times \rd$,
any linear cocycle can be identified with a matrix-valued function 
$A: \M \to GL(d,\R)$ via $F(x,v) = (f(x), A(x)v)$. 
\vskip.1cm

In this paper we take $f$ to be a transitive Anosov diffeomorphism 
of a compact manifold $\M$. However, our techniques can be applied
to hyperbolic sets and some symbolic dynamical systems. We 
consider a finite dimensional vector bundle $P : \E \to \M$ and a
H\"older continuous linear cocycle $F: \E \to \E$ over $f$
(see Section~\ref{preliminaries} for definitions). 
One of the primary examples of such 
cocycles comes from the differential $Df$ or its restriction to a H\"older 
continuous invariant sub-bundle of $T\M$. Such cocycles play a crucial 
role in smooth dynamics of hyperbolic systems. 

We establish several properties of H\"older continuous 
linear cocycles under various conditions of relative 
pinching. These properties, which include existence and continuity of
measurable invariant sub-bundles and conformal structures, 
are of independent interest and we formulate
them in the next section. As the main applications we obtain conditions
on $F$ at the periodic points of $f$ which guarantee that the cocycle 
is conformal or isometric. Our first theorem establishes a general criterium, 
and  Theorem \ref{dim 2} below gives a stronger result specific to bundles 
with 2-dimensional fibers.
We note that the assumptions of the theorems are independent of the 
choice of a continuous Riemannian metric on $\E$. 

\begin{theorem} \label{periodic} 
Let $F: \E \to \E$ be a H\"older continuous linear cocycle over 
a transitive $C^2$ Anosov diffeomorphism $f$. Suppose that there exists 
a constant $C_{per}$ such that for each periodic point $p$, 
the quasiconformal distortion satisfies
    $$
        K_F(p,n) \overset{\text{def}}{=} \,\|F^n_p\| \cdot \|(F^n_p)^{-1}\| 
        \le C_{per}       \quad \text{whenever }f^np=p.
     $$    
Then $F$ is conformal with respect to a H\"older continuous 
Riemannian metric on $\E$. 

Also, if there exists a constant $C'_{per}$ such that for 
each periodic point $p$, 
    $$
        \max\{\|F^n_p\|, \|(F^n_p)^{-1}\|\} \le C'_{per}  
        \quad \text{whenever }f^n p=p,
     $$   
then $F$ is an isometry with respect to a
H\"older continuous Riemannian metric on $\E$.

\end{theorem}

For a cocycle on a trivial bundle $\M \times \R ^d$ given by 
$A: \M \to GL(d,\R)$ the theorem implies 
cohomology to a cocycle with values 
in the conformal or orthogonal subgroup. This means that there exists 
 a H\"older continuous function $C : \M \to GL(d,\R)$ such that 
 $B(x) = C^{-1}(fx) A(x) C(x)$ is in the corresponding subgroup for all 
 $x \in \M$. The matrix $C(x)$ can be obtained as the unique positive
square root of the symmetric positive definite matrix that defines the
Riemannian metric at $x$.

Continuous reduction to orthogonal or conformal  cocycles is very useful, 
in particular, since cocycles with values in compact groups are relatively 
well understood. Some definitive results on cohomology of such cocycles 
were obtained in \cite{Liv2,Par,PP,Sch}. These results can be easily 
extended to cocycles with values in the conformal group. 
However, the question of existence of such a reduction is highly nontrivial.
Even under much stronger assumption that $\|F^n_x\|$ are uniformly
bounded for all $x\in \M$ and $n\in \Z$, the question remained open 
since it was formulated in \cite{Sch}. Under assumptions on periodic data 
only, no reduction was known until recent progress in \cite{Ka} even for 
the simplest case when $F^n_p=\Id$ for all periodic points.

Theorem \ref{periodic} can be compared to recent results by
R. de la Llave and A. Windsor \cite[Theorems 6.3, 6.8]{LW} who obtained
similar conclusions for the cocycle given by the restriction of the derivative
of an Anosov map to a H\"older continuous invariant sub-bundle of $T\M$.
The main difference is that our theorem does not have any extra assumptions
on growth or pinching of the cocycle which are present in \cite{LW}
and in most other results in the theory of non-commutative cocycles. 
\vskip.2cm

Conformality arises naturally in connection with smooth rigidity for Anosov 
systems \cite{Su,Kan,Y,L2,S,KaS}, in particular, some of our results are 
motivated by the study of derivative cocycles in \cite{KaS3}. 
It is well known that a $C^1$ small perturbation $g$ of an Anosov 
diffeomorphism $f$ is conjugate to $f$ by a H\"older homeomorphism $h$.
If $h$ is $C^1$ then $Df^n_p$ and $Dg^n_{hp}$ are conjugate by $Dh_p$
for any periodic point $p$. Therefore the conjugacy of $Df^n_p$ and 
$Dg^n_{hp}$ gives a necessary condition for $h$ to be $C^1$. This condition 
is sufficient for systems with one-dimensional stable and unstable distributions, 
but not in higher dimensions  \cite{L2}. The question of sufficiency of this 
condition is often referred to as local rigidity. Knowing that $Df$ and $Dg$ 
are conformal on the stable/unstable distribution, or on a smaller invariant distribution, helps bootstrap regularity of $h$ along the corresponding foliation.
 Thus, given certain conformality of $f$ one would like to obtain similar 
 conformality of $g$. This motivates the question whether a cocycle is
conformal given that the return maps $F^n_p$ at the periodic 
points are conjugate to conformal maps. This question is also natural from 
the point of view of cohomology of cocycles. The following proposition shows, 
however, that the answer is negative in dimension higher than two.

\begin{proposition} \label{example}
Let $f:\M\to\M$ be an Anosov diffeomorphism and $\E=\M\times \R^d$, 
$d\ge 3$. For any $\e>0$ there exists a Lipschitz continuous linear 
cocycle $F:\E\to\E$, which is $\e$-close to the identity,  such that for all
periodic points $p\in \M$ the return maps $F^n_p:\E_p\to \E_p$ are 
conjugate to orthogonal maps, but $F$ is not conformal with respect to 
any continuous Riemannian metric on $\E$.
\end{proposition} 

We note that, for a given $p$, having a uniform bound on
$\|F^n_p\| \cdot \|(F^n_p)^{-1}\|$ for all periods $n$ is equivalent to each of
the following three statements: 
$F^n_p$ is diagonalizable over $\C$ with its eigenvalues equal in modulus;
$F^n_p$ is conjugate to a conformal linear map;  
there exists an inner product on $\E_p$ with respect to which $F^n_p$ is 
conformal. 
In fact, the periodic assumption in the first part of 
Theorem \ref{periodic}  is equivalent to having such inner 
products for all periodic points uniformly bounded. 
In the context of local rigidity, additional assumptions were made to ensure 
such boundedness, for example that all return maps $F^n_p$ are  scalar 
multiples of the identity \cite{L2,KaS,L4}.
Our next result for two-dimensional bundles does not require any extra
assumptions. It can be applied, in particular, to the study of local rigidity 
without restrictive assumptions on the structure of $Df^n_p$.

\begin{theorem} \label{dim 2} 
Let $F: \E \to \E$ be a H\"older continuous linear cocycle over a transitive 
 $C^2$ Anosov diffeomorphism $f$. Suppose that the fibers of $\E$ are two-dimensional.

If for each periodic point $p\in \M$, the return map $F^n_p:\E_p\to \E_p$
is diagonalizable over $\C$ and its eigenvalues are equal in modulus,
then $F$ is conformal with respect to a H\"older continuous 
Riemannian metric on $\E$. 

Moreover, if for each periodic point $p\in \M$, the return map 
$F^n_p:\E_p\to \E_p$ is diagonalizable over $\C$ and its eigenvalues 
are of modulus 1,
then $F$ is isometric with respect to a H\"older continuous 
Riemannian metric on $\E$. 

\end{theorem}

The proof of this result overcomes essential difficulties and substantially 
differs from the proof of Theorems \ref{periodic}. 
We use Zimmer's Amenable Reduction Theorem to recast the problem as 
one of continuity of measurable invariant conformal structures and of 
measurable invariant sub-bundles. We note that such results 
on continuity of measurable invariant objects are rare beyond the case of
group valued functions with compact or abelian range.
\vskip.1cm

In the next section we formulate our main technical results.
In Section \ref{preliminaries} we briefly introduce the main notions 
used in this paper.
The proofs of all the  results are given in Section \ref{proofs}.

\vskip2cm

 \section{Properties of cocycles}

In this section we formulate our main technical results which are of 
independent interest. We consider various conditions of 
relative pinching and  establish properties of cocycles 
including existence and continuity of measurable invariant 
sub-bundles and conformal structures.
 We make the following
 \vskip.1cm
 
\noindent{\bf Standing assumptions.} {\em In  the statements below,
$f$ is a transitive $C^2$ Anosov diffeomorphism of a compact manifold $\M$,
$\;P : \E \to \M $ is a finite dimensional  H\"older continuous 
vector bundle over $\M$, and $F:\E\to\E$ is a H\"older continuous
 linear cocycle over $f$ with H\"older exponent $\beta$
 (see Section \ref{preliminaries} for definitions).}
\vskip.1cm
 
 In the first proposition we obtain uniform relative pinching of the cocycle
 from asymptotic data at the periodic points.
We denote by  $\lambda_+(F,p)$
and $\lambda_-(F,p)$ the largest and smallest Lyapunov 
exponents of $F$ at $p$, and by $\lambda_+(F,\mu)$
and $\lambda_-(F,\mu)$ the largest and smallest Lyapunov 
exponents of an ergodic invariant measure $\mu$ 
given by \eqref{exponents}.

\begin{proposition}\label{K(x,n)}
Suppose that there exists $\gamma \geq 0$ such that
$\lambda_+(F,p)-\lambda_-(F,p)\le \gamma$
for every $f$-periodic point $p\in \M$.
Then $\lambda_+(F,\mu )-\lambda_-(F,\mu )\le \gamma$ 
for any ergodic invariant measure $\mu$ for $f$. 
Moreover, for any $\e>0$ there exists $C_{\e}$ such that 
\begin{equation} \label{K(x,n) gamma}
   K_F(x,n)  \overset{\text{def}}{=} \|F^n_x\| \cdot \| (F^n_x)^{-1} \| 
   \le C_{\e} e^{(\gamma+\e) |n|} 
\quad\text{for all }x\in \M \text{ and  }n\in \Z.
\end{equation}
\end{proposition}

We can apply this proposition to the case when at each  periodic 
point $p$ there is only one Lyapunov exponent, i.e. all eigenvalues of 
$F^n_p:\E_p\to \E_p$ are of the same modulus. In this case we see 
that for any ergodic invariant measure for $f$, the cocycle 
$F$ has only one  Lyapunov exponent and \eqref{K(x,n) gamma}
is satisfied with $\gamma =0$.

\vskip.1cm

In the proposition below,  $\kappa$ is the exponent
in the Anosov condition \eqref{anosov} for $f$,
and $\beta$ is a H\"older exponent for $F$ in \eqref{beta}.
We show that under sufficient pinching, the iterates of the cocycle 
at the points on the same local stable manifold $W^s_{loc}$
remain close. The same holds for the inverse map
and the points on the same local unstable manifold $W^u_{loc}$.
To consider the compositions $(F^n_x)^{-1} \circ F^n_y$
and $(F^{-n}_x)^{-1} \circ F^{-n}_y$  we identify 
$\E_x$ and $\E_y$ for $y$ close to  $x$ using local coordinates. 
This identification is H\"older.

\begin{proposition} \label{close to Id}  
Suppose that  for some  $0 < \e < \kappa \beta /3$ there exists
$C_{\e}$ such that 
\begin{equation} \label{eq K(x,n)}
    K_F(x,n)   \le C_{\e} e^{\e |n|} \quad\text{for all }x\in \M 
    \text{ and  }n\in \Z.
\end{equation}
Then there exist $C>0$ and $\delta_0>0$ such that for any 
$\delta<\delta_0$ and $n\in \N$
\vskip.1cm

\begin{itemize}
\item[(a)] for any $x\in \M$ and $y\in W^s_{loc}(x)$  with 
$ \dist\,(x,y) \le \delta$ we have 
$$
 \|(F^n_x)^{-1} \circ F^n_y - \Id \,\| \leq C\delta^{\beta};
$$
\item[(b)]  for any $x\in \M$ and $y\in W^u_{loc}(x)$  with 
$ \dist\,(x,y) \le \delta$
we have 
$$
 \|(F^{-n}_x)^{-1} \circ F^{-n}_y - \Id \,\| \leq C\delta^{\beta}.
 $$
\end{itemize}
\end{proposition}

Next we establish continuity of measurable invariant conformal
structures and sub-bundles.  In our statements,
we consider  ergodic $f$-invariant  measures 
on $\M$ with full support and local product structure.
Examples include the measure of maximal entropy, and more 
generally Gibbs (equilibrium) measures of H\"older continuous 
potentials.
A measure $\mu$ has local product structure if it is locally
equivalent to the product of its conditional measures on the
local stable and unstable manifolds. 

\begin{proposition} \label{structure} 
Suppose that $F$ satisfies the conclusion of Proposition 
\ref{close to Id}, and $\mu$ is an ergodic $f$-invariant  measure 
on $\M$ with full support and local product structure.
Then any $F$-invariant measurable  conformal structure on $\E$ 
defined $\mu$ almost everywhere is H\"older continuous with 
exponent $\beta$. 

\end{proposition} 

It is not known in general  whether any measurable 
invariant conformal structure is continuous. Some results were 
established when the conformal structure is bounded \cite{S}
or belongs to $L^p$ for sufficiently large $p$ \cite{LS}. 

Combining Propositions \ref{K(x,n)}, \ref{close to Id}, and \ref{structure}
we see that if at each periodic point there is only one Lyapunov exponent, 
or if the largest and the smallest exponents are sufficiently close, 
then  any $F$-invariant measurable conformal structure on $\E$ 
 is H\"older continuous. 
\vskip.1cm

 We recall that a cocycle $F$ is said to be uniformly quasiconformal 
 if  the quasiconformal distortion $K_F(x,n)$ is uniformly bounded 
 for all $x\in \M$ and $n\in\Z$, see Section \ref{prelim qc} for details.
In the next proposition we apply  observations made 
by D. Sullivan \cite{Su} and P. Tukia \cite{T} for quasiconformal 
group actions to our case. We state this result in greater generality 
than our standing assumptions. We note that the converse statement 
is also true.

\begin{proposition} \label{bounded measurable}  
Let $f$ be a diffeomorphism of a compact manifold $\M$ and let
$F : \E \to \E$ be a continuous linear cocycle over $f$. If $F$ is uniformly 
quasiconformal then it  preserves a bounded measurable conformal 
 structure $\tau$ on $\E$.
\end{proposition}

Under our standing assumptions,  Proposition \ref{structure} 
implies that  $\tau$ is H\"older continuous.
We can normalize it by a  H\"older continuous
function on $\M$ to obtain a Riemannian metric with respect to which
$F$ is conformal, which yields  the following corollary.

\begin{corollary} \label{qc implies conf}
If $F$ is uniformly quasiconformal then it preserves a 
H\"older continuous conformal structure on $\E$, equivalently,
$F$ is conformal with respect to a H\"older continuous 
Riemannian metric on $\E$.

\end{corollary}

This corollary and Propositions \ref{K(x,n)} and \ref{close to Id}
enable us to prove Theorem \ref{periodic}.

\vskip.1cm

Now we address continuity of measurable invariant sub-bundles.
Note that the assumptions in the next proposition are stronger 
than those in Proposition \ref{structure}.
However, they are satisfied if at each periodic point there is only one 
Lyapunov exponent.

\begin{proposition} \label{distribution} 
Suppose that for any $\e>0$ there exists $C_{\e}$ 
such that 
  $$
    K_F(x,n)  \le C_{\e} e^{\e |n|} \quad \text{for all }x\in\M
    \text{  and  }n\in\Z.
  $$
Then any measurable $F$-invariant sub-bundle in $\E$  defined 
almost everywhere with respect to a measure with local product 
structure and full support is H\"older continuous.
\end{proposition}

Combining Propositions \ref{structure} and \ref{distribution}
with Zimmer's Amenable Reduction Theorem we obtain the
following description of cocycles with slowly growing 
quasiconformal distortion. We use it in the proof of
Theorem \ref{dim 2}. 

\begin{proposition} \label{dichotomy} 
Suppose that for any $\e>0$ there exists $C_{\e}$ such that 
$K_F(x,n)  \le C_{\e} e^{\e |n|}$ for all $x\in\M$ and  $n\in\Z$.
Then either  $F$ preserves a H\"older continuous conformal structure 
on $\E$ or  $F$ preserves a H\"older continuous proper non-trivial
sub-bundle $\E'$ of $\E$  and a
H\"older continuous conformal structure on $\E'$.
\end{proposition}

We note that the alternatives are not mutually exclusive.
If $\E'$ is one-dimentional then 
having a conformal structure on it becomes trivial.


\section{Preliminaries} \label{preliminaries}

\noindent In this section we briefly introduce the main notions 
used in this paper.

\subsection{Anosov diffeomorphisms}\label{anosov diffeo}

Let $f$ be a diffeomorphism of a compact Riemannian manifold $\M$. 
It is called Anosov if there exist a decomposition 
of the tangent bundle $T\M$ into two invariant 
continuous subbundles $E^s$ and $E^u$, and constants $C>0$, 
$\kappa>0$ such that for all $n\in \N$,
\begin{equation}\label{anosov}
\begin{aligned} 
  \| df^n(v) \|  &\leq C e^{-\kappa n} \| v \|
     \quad\text{for all }v \in E^s , \\
   \| df^{-n}(v) \| &\leq Ce^{-\kappa n} \| v \|
     \quad\text{for all }v \in E^u. 
\end{aligned}
\end{equation}
The distributions $E^s$ and $E^u$ are called stable and unstable. 
These distributions are tangential to the foliations 
$W^s$ and $W^u$ respectively. 
Local stable and unstable leaves $W^s_{loc}(x)$ and 
$W^u_{loc}(x)$ are the connected components of $x$ in the 
intersection of $W^s(x)$ and $W^u(x)$ with a small ball around $x$.

 
 \subsection{H\"older continuous vector bundles}\label{bundle} 
 
 Let $\M$ be a  compact smooth manifold.
 We consider  a finite dimensional  H\"older continuous vector bundle 
 $P : \E \to \M $ over $\M$.  By this we mean that there exists
an open cover $\{ U_i \}$ of $\M$ and a  system
 of local coordinates $\phi_i : P^{-1} (U_i) \to U_i \times \rd$ 
such that  the coordinate changes
 $$
\phi_j \circ \phi_i^{-1} : (U_i \cap U_j) \times \rd \to (U_i \cap U_j) 
\times \rd \qquad (x,v) \mapsto (x,L_x(v))
 $$ 
are homeomorphisms with liner automorphisms $L_x$ depending
 H\"older continuously on $x$. That is,  there exist $C, \;\beta >0$ 
 such that
 $$
 \| L_x - L_y \| \le C \cdot \dist (x,y) ^\beta 
 $$ 
 for all $i,j$ and all $x,y \in U_i \cap U_j$.

We will sometimes identify the fibers at nearby points using the local 
coordinates. We equip $\E$ with a background H\"older continuous  
Riemannian metric, i.e. a family of inner products on the fibers $\E_x$
depending H\"older continuously on $x$.


\subsection{Linear cocycles and Lyapunov exponents} \label{cocycles}
Let $f$ be a diffeomorphism of a compact smooth manifold $\M$
and $P : \E \to \M $ be a finite dimensional  H\"older continuous 
vector bundle over $\M$.
A H\"older continuous linear cocycle over $f$ is a homeomorphism
$F:\E\to\E$ such that $P \circ F = f \circ P$ and $F_x : \E_x \to \E_{fx}$
is a linear isomorphism which depends  H\"older continuously on $x$,
i.e. there exist $C, \,\beta>0$ such that 
for all nearby $x,y\in \M$,
\begin{equation}\label{beta}
\|F_x-F_y\|+\|F_x^{-1}-F_y^{-1}\| \le C \cdot \dist(x,y)^\beta. 
\end{equation}
Here $F_x$ and $F_y$ are viewed as matrices using local coordinates. 
Note that the second term on the left is not necessary for a continuous 
$F$.  Indeed, $F_x^{-1}$ is then automatically continuous in $x$ and bounded
on $\M$, so we can estimate 

$$
\|F_x^{-1}-F_y^{-1}\| =\| F_x^{-1}(F_y-F_x)F_y^{-1}\| \le
C' \cdot \|F_x-F_y\|.
$$
\vskip.1cm

We  consider the standard notion of Lyapunov exponents for such 
a cocycle $F$ (see \cite[Section 2.3]{BP} for more details). 
We emphasize that the Lyapunov exponents of $F$
are defined for vectors in the linear spaces $\E_x$.
Note that for any measure $\mu$ on $\M$ the vector bundle $\E$
is trivial on a set of full measure.
By Oseledets's Multiplicative Ergodic Theorem the Lyapunov exponents 
of $F$, as well as Lyapunov decomposition of $\E$, are defined almost 
everywhere for every ergodic $f$-invariant measure $\mu$ on $\M$; 
in particular, they are defined at every periodic point.  
We are primarily interested in the largest and the smallest
Lyapunov exponents of $\mu$ which can be defined as follows:
\begin{equation} \label{exponents}
\begin{aligned}
&\lambda_+(F,\mu)= \lambda_+(F,x)= \lim_{n \to \infty} \frac 1n \log \| F_x ^n \| 
\quad \text{for } \mu \text { almost every} \; x\in \M  , \\
&\lambda_-(F,\mu)= \lambda_- (F,x) = \lim_{n \to \infty} \frac 1n \log \| (F_x ^n)^{-1} \|^{-1} 
\quad \text{for } \mu \text { almost every} \; x\in \M ,
\end{aligned}
\end{equation}
\vskip.1cm
$$
\text { where} \qquad F^n_x = F_{f^{n-1}x} \circ ... \circ F_{fx} \circ F_x \hskip5.8cm
$$


\subsection{Conformal structures}\label{conf structure}

A conformal structure on $\R^d$, $d\geq 2$, is a class of proportional 
inner products. The space $\c^d$ of conformal structures on $\R^d$
identifies with the space of real symmetric positive definite $d\times d$ 
matrices with determinant 1, which is isomorphic to $SL(d,\R) /SO(d,\R)$. 
$GL(d,\R)$ acts transitively on $\c^d$
via
$$
X[C] = (\det X^TX)^{-1/d}\; X^T C \, X, \quad 
\text{ where } \; X\in GL(d,\R) \;\text{ and }\; C \in \c^d.
$$ 
It is known that $\c^d$ becomes a Riemannian symmetric space 
of non-positive curvature when equipped with a certain $GL(d,\R)$-invariant 
metric. The distance to the identity in this metric is given by
\begin{equation}\label{dist(id,C)}
\dist (\Id , C) = \sqrt{d}/2 \cdot \left( (\log \lambda_1)^2 + 
\dots +(\log \lambda_d)^2 \right) ^{1/2},
\end{equation}
where $\lambda_1, \dots, \lambda_d$ are the eigenvalues of $C$
(see \cite[p.327]{T} for more details and \cite[p.27]{M} for the formula).
The distance between two structures $C_1$ and $C_2$ can be computed 
as $\dist (C_1, C_2)= \dist (\Id, X[C_2])$,  where $ X[C_1]=\Id.$ 

It is easy to check the following relation between this metric 
and the operator norm 
\begin{equation}\label{compare}
\sqrt{d/8}  \cdot  \log (\|C\|\cdot \|C^{-1}\|) \le \dist (\Id, C)
\le d/2 \cdot \max\{ \log \|C\|,  \log \|C^{-1}\|\}.
\end{equation}
We also note that $\|C^{-1}\| \le \|C\|^{d-1}$. Thus a subset
of $\c^d$ is bounded with respect to this distance if and only if
it is bounded with respect to the operator norm.
We also note that on any bounded subset of $\c^d$ this 
distance is bi-Lipschitz 
equivalent to the distance induced by the operator norm on matrices.

\vskip.1cm

Let $\E$ be a H\"older continuous vector bundle over a compact manifold $\M$. 
A conformal structure on $\E_x$ is a class of proportional  inner products on $\E_x$.  
Using the background Riemannian metric on $\E$, we can identify an inner product
with a symmetric linear operator with determinant 1 as before. For each $x\in \M$, 
we denote the space of conformal structures on $\E_x$ by $\c(x)$.
Thus we obtain a bundle $\c$ over $\M$ whose 
fiber over $x$ is $\c(x)$. We equip the fibers of $\c$ with the Riemannian 
metric defined above.  A continuous (H\"older continuous, measurable) 
section of $\c$ is called a continuous  (H\"older continuous, measurable)
conformal structure on $\E$. 
A measurable conformal structure $\tau$ on $\E$ is called {\em bounded} 
if the distance between $\tau(x)$ and $\tau_0(x)$ is uniformly 
bounded on $\M$ for a continuous conformal structure $\tau_0$ 
on $\E$.    

Now, let $f$ be a diffeomorphism of $\M$ and $F : \E \to \E$ be a linear
cocycle over $f$. Then $F$ induces a natural pull-back action $F^\ast$ on 
conformal structures as follows. For a conformal structure 
$\tau(fx)\in \c(fx)$, viewed as the linear operator on $\E_{fx}$, 
$\;F^\ast_x(\tau(fx))\in \c(x)$ is given by
\begin{equation} \label{pull-back}
  F^\ast_x(\tau(fx))= \left( \det \, ((F_x)^T \circ F_x) \right)^{-1/n} 
  (F_x)^T \circ \tau(fx) \circ F_x,
\end{equation}
where $(F_x)^T:\; \E_{fx} \to \E_{x} $ denotes the conjugate 
operator of $F_x$. 
We note that $F^\ast_x: \c_{fx}\to \c_{x}$ is an isometry between 
the fibers $\c(fx)$ and $\c(x)$. 

We say that a conformal structure $\tau$ is $F${\em -invariant}\,
if $F^\ast(\tau) = \tau$.


\subsection{Uniform quasiconformality} \label{prelim qc}
Let $f$ be a diffeomorphism of a compact manifold $\M$ and 
$F : \E \to \E$ be a linear cocycle over $f$.
For $x\in \M$ and $n\in \Z$ the {\em quasiconformal distortion} 
of $F$ is defined by 
\begin{equation}\label{K_F}
K_F(x,n)=\frac{\max\,\{\,\|\,F_x^n(v)\,\| :\; v\in \E_x, \;\|v\|=1\,\}}
            {\,\min\,\{\,\|\,F_x^n(v)\,\| :\; v\in \E_x, \;\|v\|=1\,\}}
            = \|F^n_x\| \cdot \| (F^n_x)^{-1}  \|. 
\end{equation}
We say that $F$ is {\em uniformly quasiconformal}\, if $K_F(x,n)$
is uniformly bounded for all $x\in \M$ and  $n\in\Z $ . 
If $K_F(x,n)=1$ for all $x$ and $n$, then $F$ is said 
to be {\em conformal}. 

Clearly, $F$ is conformal with respect to a Riemannian 
metric on $\E$ if and only if it preserves the conformal 
structure associated with this metric. 
We note that the notion of uniform quasiconformality
does not depend on the choice of a continuous metric.
So if $F$ preserves a continuous conformal
structure on $\E$ then $F$ is uniformly quasiconformal on $\E$
with respect to any continuous metric on $\E$.
Corollary \ref{qc implies conf} shows that the converse is also true
if $f$ is a transitive Anosov diffeomorphism.


\section{Proofs}\label{proofs}

\subsection{Proof of Proposition \ref{K(x,n)}}

To show that $\lambda_+(\mu )-\lambda_-(\mu )\le \gamma$ 
for any ergodic invariant measure $\mu$ for $f$, we apply the 
following  theorem. 
\vskip.2cm

\noindent \cite[Theorem 1.4]{Ka}  {\it Let $f$ be a homeomorphism
of a compact metric space $X$ satisfying the closing property, 
let $F$ be a H\"older $GL(d,\R)$ cocycle over $f$, and let $\mu$ be  
an ergodic invariant measure for $f$. Then the Lyapunov exponents
 $\lambda_1 \le ... \le \lambda_d$ (listed with multiplicities) of $F$ with respect 
 to $\mu$ can be approximated 
by the Lyapunov exponents of $F$ at periodic points. More precisely,
for any $\e >0$ there exists a periodic point $p \in X$ for which the 
Lyapunov exponents $\lambda_1^{(p)} \le ... \le \lambda_d^{(p)}$ of $F$ 
satisfy $|\lambda_i-\lambda_i^{(p)}|<\e$ for $i=1, \dots , d$.}
\vskip.2cm

As stated in the remark after this theorem, it holds for any H\"older
continuous linear cocycle $F$.  Also, a transitive Anosov diffeomorphism
satisfies the closing property. Thus we can apply the theorem in our setup 
and immediately obtain the desired result for $\mu$.
\vskip.2cm

Now we prove the estimate for the quasiconformal distortion 
$K_F(x,n)$ using  the following result.
\vskip.2cm

\noindent \cite[Proposition 3.4]{RH}
{\it Let $f : \M \to \M$ be a continuous map of a compact metric space.
 Let $a_n : \M \to \R$, $n \geq 0$ be a sequence of continuous functions 
such that 
\begin{equation} \label{3.1}
    a_{n+k} (x) \le a_n (f^k (x)) + a_k (x) 
    \;\text{ for every }x \in \M,\;\; n, k \geq 0
\end{equation}    
 and such that there is a sequence of continuous functions $b_n$, 
 $n \geq 0$ satisfying 
\begin{equation} \label{3.2}
    a_n (x) \le a_n (f^k (x)) + a_k (x) + b_k (f^n (x))
     \;\text{ for every } x \in \M,\;\; n, k \geq 0.  
\end{equation}   
If  $\;\inf _n \left( \frac1n \int _\M a_n d \mu \right) < 0\;$ 
for every ergodic $f$-invariant measure, then there
is $N \geq 0$ such that 
$a_N (x) < 0$ for every $x \in \M$.}
\vskip.2cm

To simplify the notations we write $K(x,n)$ for $K_F(x,n)$.
For a given $\e >0$ we apply the proposition to 
 $$\;a_n(x)=\log K(x,n) - (\gamma+\e) n \quad \text{and} \quad
 b_n(x)=\log K(x,n) +(\gamma+\e) n.$$
It is easy to see from the definition of the quasiconformal 
distortion that  
$$
K(x, n+k) \le K(x,k) \cdot K(f^kx, n) \;\text{ and }\;
K(x, n+k) \ge K(x,n) \cdot (K(f^nx,k))^{-1} 
 $$
 for every $x \in \M,$ $n, k \geq 0$. It follows that 
$a_{n+k} (x) \le a_n (f^k (x)) + a_k (x)$, i.e. the functions $a_n$ 
satisfy  \eqref{3.1}, and $\, a_{n+k}(x) \geq a_n(x)-b_k(f^nx).$
Hence 
$$
   a_n(x)\leq a_{n+k}(x) + b_k(f^nx) 
   \le a_n (f^k (x)) + a_k (x) + b_k(f^nx)
$$
and we obtain
\eqref{3.2}.

Let $\mu$ be an ergodic $f$-invariant measure.
We note that since $a_n$ satisfy  \eqref{3.1}, 
the Subadditive Ergodic Theorem implies
that 
$$
 \inf _n \,\frac1n \int _\M a_n d \mu  \,= \lim _{n\to \infty} \,\frac 1n {a_n(x)}  
\quad \text{ for } \mu  \text{ almost all }  x \in \M.
$$
Using the definitions of $K(x,n)$,  $\lambda_+(F,\mu )$, and 
$\lambda_-(F,\mu )$ we obtain  that for $\mu$ almost all $x$
$$
\begin{aligned}
&\lim _{n\to \infty} \frac 1n \log {K(x,n)} =
\lim _{n\to \infty} \frac 1n \log (\|F^n_x\|\cdot \|(F^n_x)^{-1}\|) = \\
&\lim _{n\to \infty} \frac 1n \log \|F^n_x\| -
\lim _{n\to \infty} \frac 1n \log \|(F^n_x)^{-1}\|^{-1} =
\lambda_+(F,\mu )-\lambda_-(F,\mu )\le \gamma,
\end{aligned}
$$
and hence $\lim _{n\to \infty} \frac 1n {a_n(x)} \le -\e<0$  
for $\mu$ almost all  $x \in \M$.

Thus all assumptions of the proposition above are satisfied
and hence for any $\e >0$ there exists $N_\e$ such that 
$a_{N_\e}(x)<0$, i.e. $K(x,N_\e) \le e^{(\gamma+\e) N_\e}$ for all $x \in \M$. 
For any $n>0$, we write $n=mN_\e+r$, $0\le r<N_\e,$ and estimate
$$
\begin{aligned}
  & K(x,n) \le K(x,r)\cdot K(f^r(x),N_\e) \cdot K(f^{r+N_\e}x,N_\e)
  \cdots K(f^{r+(m-1)N_\e}x,N_\e) \\
  &  \le K(x,r) \cdot e^{(\gamma+\e) mN_\e}\le C_\e e^{(\gamma+\e) n}, 
  \end{aligned}
$$ 
where $C_\e = \max K(x,r)$ with the maximum taken over all $x \in \M$ 
and $1 \le r < N_\e$.  Since $K(x,n)= K(f^nx,-n)\,$ we obtain 
$ K(x,n) \le C_{\e} e^{(\gamma+\e) |n|}\,$ for all $x$  in $\M$ and $n$ in $\Z$.
$\QED$


\subsection{Proof of Proposition \ref{close to Id}} 

First we consider the case when $y\in W_{loc}^s(x)$.
Since at least one of the points $x$ and $y$ is non-periodic, 
we assume that $x$ is.
We denote $x_i=f^i(x)$ and $y_i=f^i(y)$ for $i=0,1,...,n$.
We have
 $$
 \begin{aligned}
 & (F^n_x)^{-1}\circ F^n_y= (F^{n-1}_x)^{-1}\circ \left( 
   (F_{x_{n-1}})^{-1} \circ F_{y_{n-1}}\right) \circ F^{n-1}_y \\
  & = (F^{n-1}_x)^{-1} \circ (\Id+r_{n-1}) \circ F^{n-1}_y
 = (F^{n-1}_x)^{-1}\circ F^{n-1}_y+(F^{n-1}_x)^{-1}\circ r_{n-1}
  \circ F^{n-1}_y \\
 & =...=\Id+\sum_{i=0}^{n-1} (F^{i}_x)^{-1}\circ r_{i}\circ F^i_y ,
 \quad \text{where  }(F_{x_i})^{-1} \circ F_{y_i}=\Id +r_i. 
 \end{aligned}
$$
We estimate
\begin{equation}\label{est}
  \|\Id-(F^n_x)^{-1}\circ F^n_y\| \le \sum_{i=0}^{n-1} 
  \|(F^{i}_x)^{-1}\| \cdot \|r_{i}\| \cdot \|F^i_y\|.
\end{equation}
Since $F$ is H\"older continuous with exponent $\b$, we have 
$$
  \|r_i\|= \|(F_{x_i})^{-1} \circ F_{y_i} - \Id \| 
  \le  \|(F_{x_i})^{-1}\| \cdot \| F_{y_i} - F_{x_i}\| 
  \le C_0\cdot\dist(x_i, y_i)^\beta.
$$
Since $y\in W_{loc}^s(x)$, for $\kappa$ is as in \eqref{anosov} we obtain 
\begin{equation}\label{r_i}
   \|r_i\|  \le C_0 (C_1\, \dist (x,y) e^{-\kappa i})^\beta 
   \le C_0 (C_1 \delta e^{-\kappa i})^\beta 
   \le C_2\delta^\beta e^{-\kappa \beta i}.
\end{equation}
Lemma \ref{F_i} below  shows that 
\begin{equation}\label{F^i}
 \| ( F^i_x )^{-1} \| \cdot  \|\, F^i_y\,\| \le C_3 e^{3i\e}
 \quad \text{for } i=0, \dots , n-1.
\end{equation}
Combining \eqref{est}, \eqref{r_i}, and  \eqref{F^i} we obtain
$$
  \|\Id-(F^n_x)^{-1}\circ F^n_y\|  \,\le\,
  \sum_{i=0}^{n-1} C_2\delta^\beta e^{-\kappa \beta i} \cdot C_3 e^{3i\e} 
   \,\le\, C_2C_3 \,\delta^\beta \,\sum_{i=0}^{n-1} ( e^{3\e -\kappa \beta})^i 
  \,\le\, C\delta^\beta 
 $$ 
since $\,3\e -\kappa \beta<0.$
This completes the proof for the case of $y\in W^s(x)$.

To prove (b) we observe that $F^{-1}$ satisfies the assumptions of 
the proposition. Indeed, 
$K_{F^{-1}}(x,n)=\|(F_x^{-n})\| \cdot \|(F_x^{-n})^{-1}\| =K_F(x,n).\;$
Thus we can apply (a) to $F^{-1}$, which yields (b).

It remains to prove estimate \eqref{F^i}.
To do this, we construct  special metrics on $\E_{f^k x}$ along the orbit 
of a  non-periodic point $x\in \M$. We denote $x_k=f^k(x)$, $k\in \Z$.

\begin{lemma} \label{metrics}
Let $f$ be a diffeomorphism of  a compact  
manifold $\M$,  $\E$ be  a continuous vector bundle over $\M$, 
and $F$ be a continuous linear cocycle over $f$. 
Suppose that for some $\e>0$ there exists $C_{\e}$ such that 
$K_F(x,n) \le C_{\e} e^{\e |n|}$ for all $x \in \M$ and  $n \in \Z$. 
Then for any non-periodic point $x \in \M$ there exist
metrics $\|\cdot \|_{x_k}$ on $\E_{x_k}$, $k\in \Z$, such that 
\begin{equation}\label{3e}
\frac{\max\,\{\,\|\,F^n_{x_k}(v)\,\|_{x_{k+n}} :\; v\in \E_{x_k}, \;\|v\|_{x_k}=1\,\}}
 {\,\min\,\{\,\|\,F^n_{x_k}(v)\,\|_{x_{k+n}} :\; v\in \E_{x_k}, \;\|v\|_{x_k}=1\,\}}
 \le e^{3|n|\e}   \quad\text{for all }k,n\in \Z.  
\end{equation}
Moreover, there exists a constant $M_\e$ such that
$\|v\| \le \|v\|_{x_k} \le M_\e\|v\|$ for all $k\in \Z$ and 
$v\in \E_{x_k}$, where $\|\cdot\|$ is a given continuous metric on $\E$.

\end{lemma}

\begin{proof} We choose a unit vector $u\in \E_x$ and set
$u_k=F^k_x(u) / \|F^k_x(u)\| \in \E_{x_k}$. 
For a vector $v\in  \E_{x_k}$ we define
$$
  \|v\|_{x_k}^2 = \sum_{m=-\infty}^\infty 
  \frac{\|F^m_{x_k} (v)\|^2}{\|F^m_{x_k} (u_k)\|^2 \cdot e^{3|m|\e}}
$$
By the assumption on $K_F$,
$\|F^m_{x_k} (\frac{v}{\|v\|})\| \cdot \|F^m_{x_k} (u_k)\|^{-1} 
\le C_\e e^{|m|\e}$
and hence the terms of this series are bounded by 
$C_\e^2 e^{-|m|\e} \|v\|^2$.
This implies that  the series converges and  
$\|v\|_{f^k x}^2 \le M_\e^2\|v\|^2,\,$
where $\,M_\e^2=C_\e^2 \sum_{n=-\infty}^\infty e^{-|m|\e}$.
Clearly, $\|v\|_{x_k}^2$ is at least the term with $m=0$,
and thus $\|v\|_{x_k} \ge \|v\|$.
\vskip.1cm

We note that it suffices to prove the estimate for $n=1$,
then it automatically follows for all $n$. 
We observe that $u_{k+1}$ is a unit vector parallel to 
$F_{x_k}(u_k)$, and hence $u_{k+1}=F_{x_k}(u_k)/\|F_{x_k}(u_k)\|$.
For any vector $v\in \E_{x_k}$ we estimate 

$$ 
\begin{aligned}
 &   \|F_{x_k}(v)\|_{x_{k+1}}^2 = \sum_{m=-\infty}^\infty 
  \frac{\|F^m_{x_{k+1}}(F_{x_k}(v))\|^2}{\|F^m_{x_{k+1}} 
  (u_{k+1})\|^2 \cdot e^{3|m|\e}} 
  = \sum_{m=-\infty}^\infty 
  \frac{\|F^m_{x_{k+1}}(F_{x_k}(v))\|^2\cdot \|F_{x_k}(u_k)\|^2}
  {\|F^m_{x_{k+1}}   (F_{x_k}(u_{k}))\|^2 \cdot e^{3|m|\e}} \\
  & = \sum_{m=-\infty}^\infty 
  \frac{\|F^{m+1}_{x_k}(v))\|^2 \cdot \|F_{x_k}(u_k)\|^2}
  {\|F^{m+1}_{x_k} (u_k))\|^2 \cdot e^{3|m|\e}} 
  =\;  \|F_{x_k}(u_k)\|^2 \sum_{j=-\infty}^\infty 
  \frac{\|F^{j}_{x_k}(v))\|^2 }
  {\|F^{j}_{x_k} (u_k))\|^2 \cdot e^{3|j-1|\e}} \\
 & \le  \|F_{x_k}(u_k)\|^2 \sum_{j=-\infty}^\infty 
  \frac{\|F^{j}_{x_k}(v))\|^2 \cdot e^{3\e}}
  {\|F^{j}_{x_k} (u_k))\|^2 \cdot e^{3|j|\e}} 
 \le  \|F_{x_k}(u_k)\|^2 \cdot \|v\|_{x_k}^2 \cdot e^{3\e}.
\end{aligned}
$$
Here we used the estimate  $|j|-1\le |j-1|$. Similarly,
using  $ |j-1|\le |j|+1$ we obtain $\|F_{x_k}(v)\|_{x_{k+1}}^2 
      \ge  \|F_{x_k}(u_k)\|^2 \cdot \|v\|_{x_k}^2 \cdot e^{-3\e} $.
Thus, for any vector $v\in \E_{x_k}$
$$
      e^{-(3/2)\e}\cdot \|F_{x_k}(u_k)\| \cdot \|v\|_{x_k} 
      \le \|F_{x_k}(v)\|_{x_{k+1}} 
      \le  e^{(3/2)\e} \cdot \|F_{x_k}(u_k)\| \cdot \|v\|_{x_k}.
$$  
It follows that for any two vectors $v,w \in  \E_{x_k}$
with $ \|v\|_{x_{k}} = \|w\|_{x_{k}}=1$
$$
   e^{-3\e} \|F_{x_k}(v)\|_{x_{k+1}} \le \|F_{x_k}(w)\|_{x_{k+1}}
   \le  e^{3\e} \|F_{x_k}(v)\|_{x_{k+1}}, 
$$   
and hence 
$ e^{-3\e}\le \|F_{x_k}(w)\|_{x_{k+1}} / \|F_{x_k}(v)\|_{x_{k+1}} 
\le e^{3\e}$.

\end{proof}

\begin{lemma} \label{F_i}
For $i=0, \dots , n-1$, 
$\;\| ( F^i_x )^{-1} \| \cdot  \|\, F^i_y\,\| \le C_3 e^{3i\e}$.
\end{lemma}

\begin{proof}
 We consider the metrics $\|\cdot\|_{x_k}$ on $\E_{x_k}$, 
 $k=0,\dots , i,$  given by  Lemma \ref{metrics}. 
 We denote by $\|F_{x_k}\|_k$ the norm of the operator  $F$ from 
 $(\E_{x_k}, \,\|\cdot\|_{x_k})$ to $(\E_{x_{k+1}}, \,\|\cdot\|_{x_{k+1}})$
 and we denote by $\|(F_{x_k})^{-1}\|_k$ the norm of the 
 corresponding inverse operator.
 Since $\|v\| \le \|v\|_{x_k} \le M_\e\|v\|$ for any $v\in \E_{x_k}$, 
 it is easy to see that 
 $(1/M_\e)\|F_{x_k}\| \le\|F_{x_k}\|_k\le M_\e \|F_{x_k}\|$.
Using   H\"older continuity of   $F$
 in the metric $\|\cdot\|$, we obtain
$$
\begin{aligned}
  \frac{\|F_{x_k}\|_k}{\|F_{y_k}\|_k} \;\le\; &
  1+ \frac{|\, \|F_{x_k}\|_k-\|F_{y_k}\|_k \,|}{\|F_{y_k}\|_k} \;\le\;
   1+ \frac{\|F_{x_k}-F_{y_k}\|_k }{\|F_{y_k}\|_k} \;\le\;
    1+ \frac{M_\e^2 \,\|F_{x_k}-F_{y_k}\| }{\|F_{y_k}\|} \\
   \;\le\;  &1+ \frac{M_\e^2 
   \cdot K_1  (\dist (x_k,y_k))^\beta}{\min_z \|F_z\|} = 
  1+ K_2 \cdot (\dist (x_k,y_k))^\beta.
  \end{aligned}
 $$ 
 
For $n=1$,  the inequality \eqref{3e} gives
 $\|F_{x_k}\|_k\cdot \|(F_{x_k})^{-1}\| _k\le e^{3\e}$
and we estimate
$$
\begin{aligned}
&\| ( F^i_x )^{-1} \|_i \cdot  \|\, F^i_y\,\|_i  \\
&\le   \|(F_x)^{-1}\|_0 \cdot \|(F_{x_1})^{-1}\|_1  \cdots 
     \|(F_{x_{i-1}})^{-1}\|_{i-1} \cdot
      \|F_y\|_0 \cdot \|F_{y_1}\|_1 \cdots  \|F_{y_{i-1}} \|_{i-1}
  \end{aligned}
$$   
$$
 \hskip1cm  =\frac{\|F_y\|_0}{\|F_x\|_0}e^{3\e} \cdot  
  \frac{\|F_{y_1}\|_1}{\|F_{x_1}\|_1}e^{3\e} 
   \cdots \frac{\|F_{y_{i-1}} \|_{i-1}}{\|F_{x_{i-1}} \|_{i-1}}e^{3\e} \; \le \;
      e^{3i\e}\prod_{k=0}^{i-1} 
      \left( 1+K_2 \cdot \left( \dist (x_k, y_k) \right)^\b \right) 
 $$ $$    
\le  e^{3i\e} \prod_{k=0}^{i-1} 
      \left( 1+K_2 \cdot \left( C_1 \delta e^{-\kappa k} \right)^\b 
      \right) \; \le  \;    e^{3i\e}
      \prod_{k=0}^{i-1} \left( 1+K_3 \cdot e^{- \b\kappa k} \right) \; \le  \; 
       K_4e^{3i\e}.  
 $$ 
 It follows that  
 $\|(F^i_x )^{-1}\| \cdot \|\, F^i_y\,\| \le M_\e^2 K_4 e^{3i\e} 
 = C_3 e^{3i\e}$.
\end{proof}     

This completes the proof of Proposition \ref{close to Id}.

\subsection{\bf Proof of Proposition \ref{structure}} 
We identify the spaces of conformal structures at nearby points by
identifying the fibers of $\E$ with $\rd$ using local coordinates.
We use the distance between conformal structures 
described in Section \ref{conf structure}.
Let $\tau$ be an invariant $\mu$-measurable conformal structure
on $\E$. 

 First we estimate the distance  between the values of $\tau$
 at $x$ and at a nearby point $y\in W_{loc}^s(x)$.
Let $x_n=f^n(x)$, $y_n=f^n(y)$, and let $D_x:=(F^n_x)^*$ be the 
isometry from $\c(f^nx)$ to $\c(x)$ induced by $F^n_x$ 
(see \eqref{pull-back}).
Since the conformal structure $\tau$ is invariant, 
$\tau(x)=D_x (\tau(x_n))$  and $\tau(y)=D_y (\tau(y_n))$. 
Using this and the fact that  $D_y$ is an isometry, we obtain
 $$
  \begin{aligned}
  \dist(\tau(x), \tau(y))&=\dist \left( D_x (\tau(x_n)), D_y (\tau(y_n)) \right) \\
  &\leq \dist \left( D_x(\tau(x_n)), D_y(\tau(x_n))  \right) +
         \dist   \left( D_y(\tau(x_n)), D_y(\tau(y_n))  \right)\\
  &=\dist \left( \tau(x_n), ((D_x)^{-1} \circ D_y)(\tau(x_n))  \right) +
         \dist  \left( \tau(x_n),\tau(y_n) \right).
 \end{aligned}
 $$
To estimate $ \left( \tau(x_n), ((D_x)^{-1} \circ D_y)(\tau(x_n))  \right)$
 we  use the following lemma.

\begin{lemma}\label{A*}
Let $\sigma$ be a conformal structure on $\R^d$ and $A$ be a linear
transformation of $\R^d$ sufficiently close to
identity. Then 
  $$
     \dist\,(\sigma, A^*(\sigma)) \le k(\sigma)\cdot\|A-\Id\,\|,
  $$
where $k(\sigma)$ is bounded on compact sets in $\c^d$. 
More precisely, if $\sigma$ is given by a matrix $C$, then
$k(\sigma) \le 3d \, \|C^{-1}\|\cdot \|C\|$
 for any $A$ with $\|A-\Id\,\| \le (6 \|C^{-1}\|\cdot \|C\|)^{-1}$. 

\end{lemma}

\begin{proof}

We write $A=\Id+R.$  Recall that the matrix $C$ corresponding 
to $\sigma$ is symmetric and positive definite with 
determinant 1. Thus there exists an orthogonal matrix $Q$ 
such that $Q^T C Q$ is a diagonal matrix whose diagonal entries 
are the  eigenvalues $\lambda_i>0$ of $C$. Let $X$ be 
the product of $Q$ and the diagonal matrix with entries 
$1/\sqrt{\lambda_i}$. Then $X$ has determinant 1
and  $X[C]=X^T C X=\Id$.  Now we estimate
$$
\begin{aligned}
  &\dist \left( \sigma, A^*(\sigma) \right) =
               \dist\, (C, A[C])=\;\dist\, (\Id, X[A[C]]) \\
  &=\dist \left( \Id, X^TA^TCAX \right)
 =\dist \left( \Id, X^T(\Id+R^T)C(\Id+R)X \right) \\&=\dist\, (\Id, \Id+B),
 \;\text{ where }\; B=X^TCRX+X^TR^TCX+X^TR^TCRX.
  \end{aligned}
$$
Since $\| R \| \le 1$, we observe that $\|B\| \le 3 \|X\|^2\cdot \|C\| \cdot \|R\|$. 
Also $\|X\|^2 \le \|C^{-1}\|$, as follows from the construction of $X$.
Thus $\|B\| \le 3 \|C^{-1}\|\cdot \|C\| \cdot \|R\|$.
Since $\| R\|\le(6 \, \|C^{-1}\|\cdot \|C\|)^{-1}$,  $\,\|B\| \le \frac 12$ 
and hence $\| (\Id +B)^{-1}\| \le 1+ 2\|B\|$.
Using \eqref{compare}, we  estimate
$$
\begin{aligned}
   \dist \, (\sigma,  A^*(\sigma) ) & =
  \dist\, (\Id, \Id+B) \le  
  d/2\cdot \log \left( \max \{ \| \Id +B \| , \| (\Id +B)^{-1}\| \} \right) \\
  & \le  d/2 \cdot \log (1+ 2 \| B \|) \le  
  d \| B \| \le 3d \, \|C^{-1}\|\cdot \|C\| \cdot \|R\| 
\end{aligned}
$$
\end{proof}

Since the conformal structure $\tau$ is $\mu$-measurable, 
by Lusin's theorem there exists a compact set $S\subset \M$ 
with $\mu(S)>1/2$ on which $\tau$ is uniformly continuous and 
bounded. 

We now show that for $x_n$ in $S$ the term 
$\,\dist \left( \tau(x_n), ((D_x)^{-1} \circ D_y)(\tau(x_n))  \right)$ 
is H\"older in $\dist (x,y)$. For this we observe that the map 
$(D_x)^{-1} \circ D_y$ is induced by $(F^n_x)^{-1}\circ F_y^n,\,$
and $\,\|(F^n_x)^{-1} \circ F^n_y - \Id \,\| \leq k\cdot \dist(x,y)^\beta$
by the assumption. 
We apply  Lemma \ref{A*} to $\sigma=\tau(x_n)$ and 
$A=(F^n_x)^{-1} \circ F^n_y$.
Since the conformal structure $\tau$ is bounded on $S$, so are 
$\|C^{-1}\|$ and $\|C\|$. We obtain that 
$$
\begin{aligned}
\dist \left( \tau(x_n), ((D_x)^{-1} \circ D_y)(\tau(x_n))  \right)
& \le \,  k(\tau(x_n)) \cdot \|(F^n_x)^{-1} \circ F^n_y - \Id \,\| \\
&\le \,  k_1\cdot \dist(x,y)^\beta,
 \end{aligned}
$$
where the constant $k_1$ depends  on the set $S$.
We conclude that if $x_n$ is in $S$ then
 $$
 \dist(\tau(x), \tau(y))
  \leq \dist(\tau(x_n),\tau(y_n)))+k_1 \cdot \dist(x,y)^\beta.
 $$ 
 
 Let $G$ be the set of points in $\M$ for which the frequency of 
 visiting $S$ equals $\mu(S)>1/2$.
By Birkhoff Ergodic Theorem $\mu(G)=1$. If both $x$ and $y$ are 
 in $G$, then there exists a sequence $\{ n_i \}$ such that 
 $x_{n_i}\in S$ and $y_{n_i}\in S$.
Since $y\in W^s_{loc}(x)$, $\,\dist(x_{n_i}, y_{n_i})\to 0$ and hence 
$\dist(\tau(x_{n_i}), \tau(y_{n_i}))\to 0$ by continuity of $\tau$ on $S$. 
Thus, we obtain
 $$
   \dist(\tau(x), \tau(y))\leq  k^s \cdot\dist(x,y)^\beta.
 $$
By a similar argument, for $x, z\in G$ with $z\in W^u_{loc}(x)$
we have 
$\dist(\tau(x), \tau(z))\leq  k^u \cdot\dist(x,z)^\beta$.

Consider a small open set in $\M$ with a product structure.
For $\mu$ almost all local stable leaves, the set of  points of 
$G$ on the leaf has full conditional measure.  
Consider points $x,y \in G$ lying on two such local stable leaves.
We denote by $H_{x,y}$ be the unstable holonomy map between 
$W^s_{loc}(x)$ and $W^s_{loc}(y)$.  Since $\mu$ has local product 
structure,  the holonomy maps are absolutely continuous
with respect to the conditional measures. Hence there exists a point 
$z\in W^s_{loc}(x)\cap G$ close to $x$ such that $H_{x,y}(z)$ is also in $G$.
By the above argument,  
  $$
   \begin{aligned}
    \dist (\tau(x), \tau(z))&\leq  k^s \cdot\dist(x,z)^\beta,  \\
    \dist(\tau(z), \tau(H_{x,y}(z))) &\leq  k^u \cdot
       \dist(z,H_{x,y}(z))^\beta, \quad\text{and} \\
    \dist(\tau(H_{x,y}(z)),\tau(y))&\leq k^s\cdot
        \dist(H_{x,y}(z),y)^\beta.
   \end{aligned}
  $$
Since the points $x$, $y$, and $z$ are close, it is clear from
the local product structure that
 $$ 
   \dist(x,z)^\beta + \dist(z,H_{x,y}(z))^\beta +
    \dist(H_{x,y}(z),y)^\beta \leq  k_5\cdot\dist(x,y)^\beta.
 $$  

Hence, we obtain $\dist(\tau(x),\tau(y))\leq  k_6
\cdot \dist(x,y)^\beta$ for  all $x$ and $y$ in a set of full measure 
$\tilde G \subset G$.
We can assume that $\tilde G$ is invariant by considering 
$\bigcap_{n=-\infty}^{\infty} f^n(\tilde G)$. Since $\mu$ has
full support, the set $\tilde G$ is dense in $\M$. Hence we can extend $\tau$ 
from $\tilde G$ and obtain an invariant H\"older continuous conformal 
structure $ \tau$ on $\M$. 
$\QED$


\subsection{Proof of Proposition \ref{bounded measurable} }

Let $\tau_0$ be a continuous conformal structure on $\E$.
We denote by $\tau_0(x)$ the conformal structure on $\E_x$, 
$x\in \M$. We consider the set 
$$
S(x)=\{\,(F^n_x)^*(\tau_0(f^n x)): \;n\in\Z\,\} 
$$ 
in $\c(x)$, the space of conformal structures on $\E_x$.
Here $F^*$ is the pull-back action given by \eqref{pull-back}. 
Since $F$ is uniformly quasiconformal,  the sets $S(x)$ have 
uniformly bounded diameters.  Since the space $\c(x)$ has 
non-positive curvature,  for every $x$ there exists a 
uniquely determined ball of the smallest radius containing $S(x)$.
We denote its center by $\tau (x)$. 

It follows from the construction that the conformal structure 
 $\tau$ is $F$-invariant and its distance from $\tau_0$ is bounded.
We also note that for any $k\ge 0$ the set 
$S_k(x)=\{\,(F^n_x)^*(\tau_0(f^n x)): \;|n| \le k \}$
depends continuously on $x$ in Hausdorff distance, and so does 
the center $\tau_k(x)$ of the smallest ball containing $S_k(x)$.
Since $S_k(x) \to S(x)$ as $k \to \infty$ for any $x$,
the conformal structure $\tau$ is the pointwise limit of continuous 
conformal structures $\tau_k(x)$. 
Hence $\tau$ is Borel measurable.
$\QED$


\subsection{Proof of Proposition \ref{distribution}} 

We consider a fiber bundle $\G$ over $\M$ whose fiber over $x$ 
is the Grassman manifold $\G _x$ of all $k$-dimensional 
subspaces in $\E_x$. The map $F_x:\E_x \to \E_{fx}$ induces 
a natural map $\F_x : \G_x \to \G_{fx}.$  
Thus we obtain a cocycle $\F : \G \to \G $ over  
$f:\M\to\M$ given by $\F(x,\xi)=( f(x), \F_x(\xi))$ where  $\xi\in \G _x$. 
Since the linear cocycle $F$ and the bundle $\E$ are H\"older 
continuous, both $\F$ and $\F^{-1}$ are  H\"older continuous and
$\dist_{C^1}(\F_x, \F_y)\le k\cdot \dist(x,y)^\beta$ for 
all $x,y\in \M$. Such $\F$ is said to be in $C^\beta (f, \G)$.
\vskip.1cm

\begin{lemma} \label{neutral_lemma}
There exists $C>0$ such that for any $x \in \M$,  
subspaces $\xi, \eta \in \G_x$, $n \in \Z$, and
$\e>0$ we have
\begin{equation} \label{neutral} 
\dist (\F^n_x (\xi), \F^n_x (\eta)) \le C\cdot K_F(x,n) \cdot \dist (\xi, \eta)
\le C\cdot C_\e e^{\e |n|}\cdot \dist (\xi, \eta).
\end{equation}
\end{lemma}

\begin{proof}
Let $w$ and $v$ be two unit vectors in $\E_x$.
We denote $D=F^n_{ x}$. Using the formula
$\; 2<Dw, Dv>=\|Dw\|^2+\|Dv\|^2-\|Dw-Dv\|^2$ for the inner product, 
we obtain 
$$
(2 \, \sin ({\angle (Dw, Dv)}/2)\,)^2 =2 (1-\cos \angle (Dw, Dv) )=
$$
$$
  2-\frac{2 <Dw, Dv>}{\|Dw\|\cdot \|Dv\|}=
  \frac{2 \|Dw\|\cdot \|Dv\| - \|Dw\|^2 - \|Dv\|^2 + \|Dw-Dv\|^2} 
  {\|Dw\|\cdot \|Dv\|} =
$$  
$$
  \frac{ \|Dw-Dv\|^2 -(\|Dw\| - \|Dv\|)^2 } { \|Dw\|\cdot \|Dv\|} \le
  \frac{ \|D\|^2 \cdot \|w-v\|^2} { \|Dw\|\cdot \|Dv\|} \le K_F(x,n) ^2  
  \cdot \|w-v\|^2. 
$$ 
Suppose that the angle between the unit vectors $w$ and $v$ is 
sufficiently small so that it remains small when multiplied by $K_F(x,n)$. 
Then we obtain that the angle $\angle (Dw, Dv)$ is also small and
$$
\angle (Dw, Dv) \le C_0\cdot K_F(x,n) \cdot \angle (w, v)
$$
If the right-hand side is large the estimate is trivial, 
thus for any subspaces $\xi,\eta \in \G _x$ we have
$$ 
\dist (\F^n_x (\xi), \F^n_x (\eta)) \le C\cdot K_F(x,n) \cdot \dist (\xi, \eta),
$$
where the distance between two subspaces is the maximal angle. 
 We note that the 
maximal angle distance is Lipschitz equivalent to any smooth 
Riemannian metric on the Grassman manifold, and thus 
we have the estimate in any smooth metric on $\G$. 

The case of $n<0$ can be considered similarly.
\end{proof}

The lemma implies that the expansion/contraction in the fiber
is arbitrarily slow and, in particular, slower than the expansion/contraction
of the hyperbolic system in the base. Hence the cocycle $\F$ is
{\em dominated} in the sense of  \cite[Definition 4.1]{ASV}. This notion is 
similar to the notion of domination or bunching for partially hyperbolic 
systems, the difference is that in our context $\F$ is not a diffeomorphism,
only the maps $\F_x$ are.

Dominated cocycles have H\"older continuous strong (un)stable foliations.
They sub-foliate the weak (un)stable leaves which are preimages of
(un)stable leaves in the base. The strong stable foliation gives rise to 
an {\em $s$-holonomy} for $\F$, an invariant family of maps between 
the fibers over the same stable leaf in the base. These facts are conveniently 
summarized in the following proposition, whose Lipschitz version appeared
in \cite[ Proposition 4.1]{AV}. We will refer only to part (3), which in our 
setting can be easily obtained from Proposition \ref{close to Id} and
its proof. Indeed, the desired holonomy $H^s_{x,y}$ is induced on the 
Grassmannians by $\lim_{n \to \infty} (F^n_x)^{-1} \circ F^n_y$.

\vskip.2cm

\noindent\cite[Proposition 4.2]{ASV} 
{\it If the cocycle $\F\in C^\beta(f,\G)$ is dominated
then there exists a unique partition $\W^s=\{ \W^s(x,\xi):\; (x,\xi)\in \G\}$
of the fiber bundle $\G$ such that
\vskip.1cm
\begin{enumerate}
\item every  $\W^s(x,\xi)$ is a $\beta$-H\"older graph over $W^s(x)$, 
          with H\"older constant uniform on $x$;
         \vskip.1cm
\item $\F(\W^s(x,\xi))\subset \W^s(\F(x,\xi))$ for all $(x,\xi) \in \G$;
            \vskip.1cm
\item the family of maps $H^s_{x,y}: \G _x \to \G_y$ defined for 
        $y \in W^s(x)$ by 
         $(y,H^s_{x,y}(\xi))\in \W^s(x,\xi)$ is an  $s$-holonomy for $\F$.
\end{enumerate}
Moreover, there is a dual statement for strong unstable leaves.
}
\vskip.2cm

Let $\mu$ be a measure on $\M$ with local product structure 
and full support.
A $\mu$-measurable $F$-invariant sub-bundle in $\E$ gives rise to a 
measurable $\F$-invariant section $\phi:\M \to \G$. 
We denote by $m$ the lift of $\mu$ to the graph $\Phi$ of $\phi$, i.e.
for a set $X\subset \G$, $\;m(X)=\mu(\pi(X\cap \Phi))$, where 
$\pi : \G \to \M$ is the projection. Alternatively, $m$ can be
defined by specifying that  for $\mu$-almost every $x$ in $\M$ the conditional 
measure $m_x$ in the fiber $\G _x$ is the atomic measure at $\phi(x)$.
Since $\mu$ is $f$-invariant and $\Phi$ is $\F$-invariant,
the measure $m$ is $\F$-invariant. 

Lemma \ref{neutral_lemma} implies that  Lyapunov
exponent of  $\F$ along the fiber is zero at every $\xi \in \G$, 
in particular the exponent of $m$ along the fiber is zero. 
This together with the existence of $s$- and $u$-holonomies for $\F$ 
allows us to apply \cite[Theorem C]{AV}
to the measure $m$ and conclude that there exists a system of 
conditional measures $\tilde m_x$ on $\G_x$ for $m$ which are 
holonomy invariant and vary continuously on $x$ in supp$\,\mu = \M$. 
Since the systems of conditional measures $\{m_x\}$ and 
$\{\tilde m_x\}$ coincide on a set $X$ of full $\mu$ measure, we see that
$\tilde m_x = m_x$ is the atomic measure at $ \phi(x)$ for all $x\in X$. 
Since $X$ is dense we obtain that $\tilde m_x$ is atomic for all $x\in \M$. 
Indeed, by compactness of $\G$, for any $x\in \M$ 
we can take a sequence $X \ni x_i  \to x$ so that $\phi (x_i)$ converge 
to some $\xi \in \G_x$. This implies that $\tilde m_{x_i}=m_{x_i}$ converge to 
the atomic measure at $\xi$, which therefore coincides with $\tilde m_x$ 
by continuity of the family $\{\tilde m_x\}$. 
Denoting $\tilde \phi(x) = \text{supp } \tilde m_x$ for $x\in \M$, 
we obtain a continuous section $\tilde \phi$ which  coincides with 
$\phi$ on $X$. Since $\tilde \phi$ is invariant under $s$- and $u$-holonomies
we conclude that it is $\beta$-H\"older.
This yields that the invariant measurable sub-bundle 
in $\E$ coincides $\mu$-almost everywhere with a H\"older continuous one. 
$\QED$


\subsection{Proof of Proposition \ref{dichotomy}}

We use the following particular case of Zimmer's Amenable 
Reduction Theorem:
\vskip.2cm

 \noindent 
 \cite[Corollary 1.8]{HK},  \cite[Theorem  3.5.9]{BP}
 {\it Let  $f$ be an ergodic transformation of a measure space 
 $(X,\mu)$ and let $F: X \to GL(d,\R)$ be a measurable function.
 Then there exists a measurable function $C :X \to GL(d,\R)$ 
 such that the function $G(x) = C^{-1}(fx) F(x) C(x)$ takes
 values in a maximal amenable subgroup of $GL(d,\R)$.}

\vskip.2cm 
 
 It is known that any maximal amenable subgroup of $GL(d,\R)$
 is conjugate to a group of  block-triangular matrices of the form
 $$
 \left( \begin{array}{cccc}
 A_1 & \ast & \ldots & \ast \\
 0  & A_2 & \ddots &\vdots \\
 \vdots & \ddots & \ddots & \ast \\
 0 & \ldots & 0 & A_r
 \end{array} \right)
 $$
 where each  diagonal block $A_i$ is a scalar
 multiple of a $d_i\times d_i$ orthogonal matrix 
 and $d_1+\dots + d_r=d$.
 
\begin{corollary} \label{dich} 
Let $f$ be a diffeomorphism of a smooth compact manifold 
$\M$ preserving an ergodic measure $\mu$ and let $F:\E \to \E$ 
be a measurable linear cocycle over $f$. 
Then $F$ preserves a measurable
conformal structure  either on $\E$ or  on a measurable invariant
proper non-trivial sub-bundle of $\E$.
\end{corollary}

\begin{proof}
We recall that $\E$ can be trivialized on a set of full $\mu$-measure
\cite[Proposition 2.1.2]{BP}, so we measurably identify $\E$ with 
$\M \times \rd$ and  view $F$ as a function $\M \to GL(d,\R)$. 
Thus can we apply the Amenable 
Reduction Theorem to $F$ and obtain a measurable coordinate change 
function $C :\M \to GL(d,\R)$ such that $G_x = C^{-1}(fx) F_x \,C(x)$ is of 
the above form for $\mu$-almost all $x\in \M$. If $r=1$ and $d_1=d$ we obtain
that $G_x$ is a scalar multiple of a $d\times d$ orthogonal matrix.
This implies that $F$ is conformal with respect to the pull back by $C^{-1}$
of the standard conformal structure on $\rd$. This gives a measurable 
invariant conformal structure for $F$ on $\E$. If $r>1$ then the last 
block $A_r$ acts conformally on a $d_r$-dimensional invariant 
subspace for $G$. Again pulling back by $C^{-1}$ we obtain a measurable 
invariant sub-bundle with conformal structure for $F$.
\end{proof}

Now suppose that that the system satisfies the assumptions 
of  Proposition \ref{dichotomy}. We apply Corollary \ref{dich}
with $\mu$ being the measure of maximal entropy for $f$.
If we have a measurable invariant conformal structure on $\E$
then it is H\"older continuous by Proposition \ref{structure}.
If we have a measurable invariant sub-bundle $\E'$ then it is 
H\"older continuous by Proposition \ref{distribution}. The restriction
of $F$ to $\E'$ is a H\"older continuous cocycle which also satisfies
the same pinching assumption as $F$. Hence the invariant measurable 
conformal structure on $\E'$ is again H\"older continuous by 
Proposition \ref{structure}.
$\QED$


\subsection{Proof of Theorem \ref{periodic}}

By Proposition \ref{K(x,n)}, the assumption of the theorem
implies that \eqref{K(x,n) gamma} is satisfied with $\gamma =0$.
Hence the conclusions of Proposition \ref{close to Id} hold,
which has the following implication for
quasiconformal distortion.

\begin{lemma} \label{distort}
Let $A$ and $B$ be two linear transformations of $\,\R^d$.
Suppose that either $\|A^{-1}B-\Id\,\|\le r $ or 
$\|AB^{-1}-\Id\,\|\le r$, where $r<1$. Then 
$$
    (1-r)/(1+r) \le K(A)/K(B)\le (1+r)/(1-r),
$$
where $K(A)$ and $K(B)$ are quasiconformal distortions
of $A$ and $B$ respectively.
\end{lemma}

\begin{proof} 
We recall that $K(A)=K(A^{-1})$ and $K(A_1A_2)\le K(A_1)K(A_2)$.

Suppose that $\|A^{-1}B-\Id\,\|\le r $. Denoting $A^{-1}B-\Id =R$
and multiplying by $A$, we have $B=A(\Id+R)$.
Since for any unit vector $v$, $\, 1-r\le \|(\Id+R)v\| \le 1+r$,
we obtain
   $$
     K(B)\le K(A)K(\Id+R) \le K(A)\cdot(1+r)/(1-r).
   $$
Multiplying by $B^{-1}$, we have $A^{-1}=(\Id+R)B^{-1}$ and hence
   $$
    K(A)=K(A^{-1})\le K(\Id+R) K(B^{-1}) \le (1+r)/(1-r)\cdot K(B).
   $$
The case of $\|AB^{-1}-\Id\,\|\le r$ is similar.
\end{proof}

Now we will show that $F$ is uniformly quasiconformal using a 
dense orbit argument.
Since $f$ is transitive, there exists a point $z\in \M$ with dense 
orbit $\o =\{ f^k z ;\; k \in \Z\}$. We will show that the quasiconformal 
distortion $K_F(z,k)$ is uniformly bounded in $k \in \Z$. Since $\o$ 
is dense and $K_F(x,k)$  is continuous on $\M$ for each $k$,
this implies that $K_F(x,k)$ is uniformly bounded 
in $x \in X$ and $k \in \Z$.

We consider any two points of $\o$ with 
$\dist(f^{k_1}z, f^{k_2}z)< \delta_0$,
where $\delta_0$ is sufficiently small to apply the Anosov Closing 
Lemma \cite[Theorem 6.4.15]{KH}. 
We assume that $k_1<k_2$ and denote $x=f^{k_1}z$ and 
$n=k_2-k_1$, so that $\delta =\dist(x, f^n x)< \delta_0$. 
By the Anosov Closing Lemma there exists $p\in X$ with $f^n p =p$ 
such that $\dist(f^i x, f^i p) \le c\delta$ for $i=0, \dots , n$.

Let $y$ be  a point in  $ W^s_{loc}(p)\cap W^u_{loc}(x)$.
Then by Proposition \ref{close to Id},
$$
 \|(F^n_p)^{-1} \circ F^n_y - \Id \,\| \leq C\delta^{\beta}
\quad\text{and}\quad
 \|(F^{-n}_y)^{-1} \circ F^{-n}_x - \Id \,\| \leq C\delta^{\beta}.
 $$
Hence  by Lemma \ref{distort}
$$
   K_F(y,n)/K_F(p,n) \le (1+C\delta^\beta)/(1-C\delta^\beta)\le 2 
   \quad\text{and}\quad K_F(x,n)/K_F(y,n) \le 2
   $$
if $\delta_0$ is sufficiently small. 
Thus $K_F(x,n) \le 4K_F(p,n) \le 4C_{per}$.

We take $m>0$ such that the set  $\{f^j(z);\; |j|\le m\}$ is 
$\delta_0$-dense in $\M$. Let
$K_m =\max\{K_F(z,j):\; |j|\le m\}$. Then for any $k>m$
there exists $j$, $|j|\le m$, such that $\dist (f^k(z), f^j(z))\le \delta_0$
and hence 
  $$
   K_F(z,k) \le K_F(z,j)\cdot K_F(f^j(z), k-j) \le K_m \cdot 4C_{per}.
  $$ 
The case of $k<-m$ is considered similarly. Thus $K_F(z,k)$ is 
uniformly bounded and hence so is $K_F(x,k)$ for all $x\in \M$ 
and $k\in \Z$.

Thus, $F$ is uniformly quasiconformal on $\E$. It follows from 
Corollary \ref{qc implies conf} that there exists a H\"older 
continuous metric on $\E$ with respect to which $f$ is conformal.
\vskip.2cm

Now we prove the second part of the theorem.
We observe that 
$K_F(p,n)=\|F^n_p\| \cdot \|(F^n_p)^{-1}\| \le (C_{per}')^2$
and hence  $\,F$ is conformal with respect 
to a H\"older continuous Riemannian metric $g$ on $\E$.
This means that there exists a positive H\"older continuous function 
$a(x)$  such that for each  $x\in\M$ and each $v\in \E_x$
$$
   \|F_x(v)\|_{g(fx)} = a(x) \cdot \|v\|_{g(x)}.
 $$  

It remains to renormalize the metric $g$. 
For a positive function $\varphi(x)$, we consider a new metric
$\tilde g (x)= g(x)/\varphi(x)$.
Then we have 
$$
   \varphi(fx) \cdot \|F_x(v)\|_{\tilde g(fx)} = 
   a(x)\varphi(x) \cdot \|v\|_{\tilde g(x)}.
 $$  
Therefore  we need to find a H\"older continuous
function $\varphi(x)$ such that $\varphi(fx)= a(x)\varphi(x)$, 
i.e. 
\begin{equation}\label{a(x)}
   a(x)=\varphi(fx)/\varphi(x) \quad \text{for all }x\in\M.
\end{equation}

Let $p$ be a periodic point of a period $n$. Then 
$\,a(p)a(fp)\dots a(f^{n-1}p)=\|F^n_p\|_{g(p)}$.
If $\|F^n_p\|_{g(p)} > 1$ then 
$\|F^{mn}\|_{g(p)}=\|F^n_p\|_{g(p)}^m \to \infty $ 
as $m\to\infty$. If $\|F^n_p\|_{g(p)} < 1$ then 
$\|(F^{mn})^{-1}\|_{g(p)}=\|(F^n_p)^{-1}\|_{g(p)}^m = 
\|F^n_p\|_{g(p)}^{-m}\to \infty$ as $m\to\infty$. 
Either case contradicts the assumption of the proposition and hence 
$a(p)a(fp)\dots a(f^{n-1}p)=1$ for any periodic point $p$.
Now, Liv\v{s}ic Theorem \cite[Theorem 19.2.1]{KH}
implies that the equation \eqref{a(x)} has a H\"older
continuous solution $\varphi(x)$, and  $F$ is an isometry
with respect to the H\"older continuous metric $\tilde g$.
$\QED$


\subsection{Proof of Theorem \ref{dim 2}}
By Proposition \ref{K(x,n)}, the assumption on the periodic data 
implies that \eqref{K(x,n) gamma} is satisfied with $\gamma =0$.
Now Proposition \ref{dichotomy}  yields that $F$ preserves either a  
H\"older continuous one-dimensional sub-bundle in $\E$ 
or a H\"older continuous invariant conformal structure on $\E$.

Suppose that there is a continuous invariant one-dimensional sub-bundle. 
Hence for any point $p$ and any $n$ with $f^np=p$ we obtain an invariant 
line for $F^n_p:\E_p \to \E_p$. This implies that the eigenvalues 
of $F^n_p$ are real. Hence, by the assumption of the theorem,
they are either $\lambda ,\lambda$ or $\lambda ,-\lambda$. 
If $F$ is orientation preserving, the former is always the case.
It follows that  $F^n_p = \lambda \cdot \Id$ since it is 
diagonalizable,  and hence $K_F(p,n)=1$.  For such $F$ we can 
apply Theorem \ref{periodic} and obtain a H\"older continuous metric
on $\E$ with respect to which $F$ is conformal. 

If $\E$ is not orientable we can pass to a double cover. If
$F$ is orientation reversing, we can consider cocycle $F^2$ 
over $f^2$. Thus we can always obtain an orientation preserving
cocycle $F'$ which by the above is conformal. This implies
uniform quasiconformality of the original cocycle $F$.
Now Corollary \ref{qc implies conf} yields conformality of $F$.

We conclude that $F$ is conformal with respect to a H\"older 
continuous metric on $\E$. The second part can be establishes 
in the same way as in the proof Theorem \ref{periodic}. Indeed, 
the assumption implies that for any periodic point $p$ the map $F^n_p$
is conjugate to an orthogonal matrix and hence there exists a 
constant $C(p)$ such that $\max\{\|F^n_p\|, \|(F^n_p)^{-1}\|\} \le C(p)$ 
for any period $n$ of $p$.
$\QED$

\vskip.2cm


\subsection{Proof of Proposition \ref{example}}

The idea of the example was 
suggested to us by M. Guysinsky and is similar to his example in \cite{Guy}.
We recall that since the tangent bundle is trivial, a cocycle $F:\E\to \E$ 
is defined by a function $A:\M\to GL(d,\R)$ via
$F(x,v)=(fx, A(x)v).$ First we construct an example for $d=3$.
\vskip.1cm

Let $S$ be a closed, and hence compact,  $f$-invariant set in $\M$ 
that does not contain any periodic points (such sets always exist for
Anosov systems and can be constructed using symbolic dynamics).
Let $\a : \M\to \R$ be given by 
$$
   \a(x)=\dist (x,S)  \cdot \e / (2\,\diam\M).
$$ 
Then $\a$ is Lipschitz, 
$$
    \a(x)=0 \; \text{ for all }x\in S \quad \text{and}\quad
    0< \a(x) \le\e /2  \; \text{ for all }x\notin S.
$$
We set 
$$ 
     A(x)=\left( \begin{array}{ccc}
             \cos \ta(x)  & -\sin \ta(x)       & \; \e       \\
             \sin \ta(x)   & \;\;\;\cos \ta(x) & \; 0         \\
             0               & 0                    & \; 1    \\             
           \end{array} \right) ,
$$
where $\ta$ is a modification of the function $\a$ constructed below.
For a point $p\in \M$ and $n\in \N$ we denote 
$$
\begin{aligned}
    A(p,n)&=A(f^{n-1}p) \cdots A(fp) \cdot A(p), \\
   \ta(p,n)&=\ta(f^{n-1}p) + \dots+ \ta(fp)+\ta(p).
 \end{aligned}
$$
Then we have
$$
      A(p,n) = \left( \begin{array}{ccc}
             \cos \ta(p,n)  & -\sin \ta(p,n)       & \; \ast       \\
             \sin \ta(p,n)   & \;\;\;\cos \ta(p,n) & \; \ast         \\
             0               & 0                    & \; 1    \\             
           \end{array} \right). 
 $$

 Let $p$ be a periodic point and let $n$ be its minimal period. 
Since the eigenvalues of $A(p,n)$ are of modulus $1$, it is conjugate 
to an orthogonal matrix if and only if it is diagonalizable over $\C$.
For $A(p,n)$ to be diagonalizable, it suffices to have three different 
eigenvalues, which is equivalent to $\ta(p,n)\ne \pi k$. 
The function $\a$ does not necessarily satisfy this condition at every 
periodic point, so we modify it inductively
to obtain a function $\tilde \a$ that does.

Since there are countably many periodic orbits for $f$,
we can order them $\{ \O_1, \O_2, \dots \}$.
Let $\a_0=\a$ and suppose that $\a_{m-1}$ is defined.
If $\a_{m-1}(p, n_m)\ne \pi k$, where $p\in \O_m$ and $n_m$ is
the minimal period of $p$, we set $\a_m=\a_{m-1}$.
If $\a_{m-1}(p, n_m)= \pi k$ we modify $\a_{m-1}$ 
in a small neighborhood of $\O_m$ to obtain $\a_m$.
Let 
$$
  \delta_m = \dist \,
(\O_m, \, \O_1 \cup \dots \cup \O_{m-1}\cup S) / 2.
$$
We change $\a_{m-1}$ in the $\delta_m$-neighborhood of $\O_m$
so that the new function $\a_m$ is Lipschitz on $\M$,
$\a_m(p, n_m)\ne \pi k$, and the Lipschitz norm of $\a_{m-1} -\a_m$ is
less than $\e/(4m^2)$.
Clearly $\a_m=\a_{m-1}$ on $S$ and on $\O_1, \dots ,\O_{m-1}$.
We define $\ta(x)=\lim_{m \to \infty} \a_m(x)$.

It follows from the construction that $\ta(x)$ is Lipschitz,
$\ta(x)=0$ on $S$, and $\ta(p, n)\ne \pi k$  for any periodic point $p$ 
of a minimal period $n$. Thus the matrix $A(p,n)$ has three different eigenvalues:
$e^{i\ta(p,n)}$, $e^{-i\ta(p,n)}$, 1, and hence is diagonalizable. 
We also note that since $\max |\a(x)-\ta(x)| \le \e/2$, we have 
$\max |\ta(x)| \le \e$ and 
thus $A(x)$ is $\e$-close to the identity. 

\vskip.1cm

Now we show that the cocycle $F$ is not uniformly quasiconformal.
Let $x$ be a point in $S$. Then $x$ is non-periodic and 
$\ta(f^nx)=0$ for every $n$.
Hence 
$$
     A(x) = \left( \begin{array}{ccc}  1 & 0 & \e \\ 0 & 1 & 0 \\ 0& 0 &1
     \end{array} \right) \quad \text{and}\quad
    A(x,n) = \left( \begin{array}{ccc}  1 & 0 & n\e \\ 0 & 1 & 0 \\ 0& 0 &1
     \end{array} \right). 
$$ 
The quasiconformal distortion of $F$ is not uniformly bounded
along the orbit of $x$, as for the first coordinate vector
$\|F_x^n v_1\|=\|A(x,n)v_1\|\to\infty$, while for the second coordinate 
vector $\|F_x^n v_2\|=\|A(x,n)v_2\|=1$. It follows that $F$ cannot 
be conformal with respect to any continuous Riemannian metric on $\E$.
\vskip.1cm

This example can be extended to any dimension $d\ge 4$
by considering 
$$
   F(x,v)=(fx, \tilde A (x)v) \quad \text{with}\quad
     \tilde A (x) = \left( \begin{array}{cc} A(x) & 0 \\ 0 & \Id_{d-3}
     \end{array} \right),
$$ 
where $\Id _{d-3}$ is the $(d-3)\times (d-3)$ identity matrix.
$\QED$


\end{document}